\documentclass[a4paper,11pt,leqno]{amsart}
\usepackage{amsmath}
\usepackage{amsthm}
\usepackage{amsfonts}
\usepackage{amssymb}

\usepackage{esint}

\usepackage[usenames]{color}
\usepackage[all]{xy}

\usepackage{url}
\usepackage[all]{xy}
\usepackage{graphicx}
\usepackage{latexsym}
\usepackage[cp850]{inputenc}
\usepackage{epsfig}
\usepackage{psfrag}
\usepackage{dsfont}

\newtheorem{theorem}{Theorem}[section]{\bf}{\it}
\newtheorem{lemma}[theorem]{Lemma}{\bf}{\it}
{\bf}{\it}
\newtheorem{corollary}[theorem]{Corollary}{\bf}{\it}
{\bf}{\it} % Main Theorems, numbered A,B,...
{\bf}{\it}
\newtheorem*{theorem*}{Theorem}

\newtheorem*{namedtheorem}{\theoremname}
\newcommand{\theoremname}{testing}

{\bf}{\it}
{\bf}{\it}

\theoremstyle{remark}
\newtheorem{remark}{Remark}%[theorem]{Remark}{\bf}{\it}

\theoremstyle{definition}
\theoremstyle{remark}

\numberwithin{equation}{section}

\newcommand{\R}{\mathbb R}

\newcommand{\N}{\mathbb N}

\newcommand{\loc}{{\operatorname{loc}}}

\newdimen\vintkern\vintkern11pt
\def\vint{-\kern-\vintkern\int}

\newcommand{\vol}{\mathrm{vol}}

\newcommand{\Cone}{\mathrm{Cone}}
\newcommand{\bG}{\mathbb{G}}

\newcommand{\G}{\mathbb{G}}

\begin{document}

\title[BLD-ellipticity implies Abelian fundamental group]{Closed BLD-elliptic manifolds have virtually Abelian fundamental groups}
\date{November 13, 2013}
\author{Enrico Le Donne \and Pekka Pankka}
\address{Department of Mathematics and Statistics, P.O. Box 35, FI-40014 University of Jyv\"askyl\"a, Finland}
\email{\{enrico.ledonne,pekka.pankka\}@jyu.fi}
 
%\thanks{The second author is supported by the Academy of Finland project \#256228.}

\begin{abstract}
We show that a closed, connected, oriented, Riemannian $n$-manifold, admitting a branched cover of bounded length distortion from $\R^n$, has a virtually Abelian fundamental group.
\end{abstract}


\maketitle

\section{Introduction}
A continuous, discrete and open map $M\to N$ between $n$-manifolds is called a \emph{branched cover}. A branched cover $f\colon M\to N$ between Riemannian $n$-manifolds has  \emph{bounded length distortion} if $f$ is \emph{bilipschitz on paths}, that is, there exists a constant $L\ge 1$ for which
\[
L^{-1}\,\ell(\gamma) \le \ell(f\circ \gamma) \le  L \,\ell(\gamma)
\]
for all paths $\gamma$ in $M$, where $\ell(\cdot)$ is the length of a path.
We call branched covers of bounded length distortion just {\em BLD-maps}. 
Recall that in case $M$ is complete and $N$ is connected, any BLD-map is surjective.  

A BLD-map, as defined here, between oriented Riemannian manifolds is either orientation preserving or orientation reversing by \v{C}ernavski\u{\i}--V\"ais\"al\"a theorem; see \cite{VaisalaJ:Disopm}. 
Orientation preserving BLD-maps were first considered by Martio and V\"ais\"al\"a in \cite{MartioO:Ellemb} as a strict subclass of quasiregular maps and the metric theory of BLD-maps was developed in detail by Heinonen and Rickman in \cite{HeinonenJ:Geobcg}. 
Recall that a continuous map $f\colon M\to N$ between oriented Riemannian $n$-manifolds is \emph{quasiregular} if $f$ belongs to the local Sobolev space $W^{1,n}_\loc(M,N)$ and satisfies the \emph{quasiconformality condition}, that is, there exists $K\ge 1$ for which 
\[
|Df|^n \le K J_f,\ \text{ almost everywhere.}
\]
Here $|Df|$ is the operator norm of the differential $Df$ of $f$ and $J_f$ the Jacobian determinant. A connected, oriented, and Riemannian $n$-manifold $N$ is \emph{quasiregularly elliptic} if there exists a non-constant quasiregular map $\R^n \to N$. 

In this note, we consider fundamental groups of closed {\em BLD-elliptic manifolds}, that is, closed, connected and Riemannian $n$-manifolds $N$, for $n\ge 2$, which admit a BLD-map $\R^n \to N$. By a theorem of Varopoulos \cite[pp. 146-147]{VaropoulosN:Anagog}, the fundamental group of a closed quasiregularly elliptic $n$-manifold has growth of polynomial order at most $n$. By Gromov's theorem on groups of polynomial growth, the fundamental group is therefore virtually nilpotent. It is a natural question whether we can say more on the structure of such a  fundamental group.

In \cite{LP}, Luisto and the second-named author showed that closed quasiregularly elliptic manifolds having maximal order of growth have virtually Abelian fundamental group. For BLD-elliptic manifolds, no such additional condition is needed.

\begin{theorem}
\label{thm:main}
Let $N$ be a closed, connected and Riemannian $n$-manifold admitting a BLD-map $\R^n \to N$. Then $\pi_1(N)$ is virtually Abelian.
\end{theorem}

The proof is based on the reinterpretation of BLD-maps as Lipschitz quotients introduced by Bates, Johnson, Lindenstrauss, Preiss, and Schechtmann in \cite{BatesS:AffaLf}. A map $f\colon X\to Y$ between metric spaces is an \emph{$L$-Lipschitz quotient} for $L\ge 1$ if 
\[
B_Y(f(x),r/L) \subset f(B_X(x,r)) \subset B_Y(f(x),Lr)
\]
for all $x\in X$ and $r>0$, where $B_X(x,s)$ and $B_Y(y, t)$ are metric balls about $x\in X$ and $y\in Y$ and of radii $s>0$ and $t>0$ in $X$ and $Y$, respectively. A standard path-lifting argument shows that BLD-maps are Lipschitz quotients; see e.g.\;\cite[Proposition 4.13]{HeinonenJ:Geobcg}. The converse, whether Lipschitz quotients between Riemannian $n$-manifolds are BLD-maps, is true for $n=2$ and is an intriguing open question in geometric mapping theory in higher dimensions, see \cite[Section 4]{BatesS:AffaLf} or \cite[Section 4]{HeinonenJ:Geobcg}. 

Our method gives, in fact, a version of Theorem \ref{thm:main} for Lipschitz quotients.
\begin{theorem}
\label{thm:LQ_main}
Let $n,m\in \N$.
Let $N$ be a closed, connected, and Riemannian $m$-manifold admitting a  Lipschitz quotient  $\R^n \to N$. Then $\pi_1(N)$ is virtually Abelian and has polynomial order of growth at most $n$.
\end{theorem}

Note that in the last theorem we are not assuming anymore that the dimension of $\R^n$ is the same of the dimension of $N$.
Our method is based on ultralimits. Due to a possible change of dimension, the BLD-condition does not pass to ultralimits. However, the ultralimits of Lipschitz quotients are Lipschitz quotients. Thus, by considering a blow-down of the universal cover and the fundamental group, we find a Lipschitz quotient map from $\R^n$ to a Carnot group $\bG$ which is the asymptotic cone of $\pi_1(N)$. Then, by passing to a tangent i.e.\;by considering a blow-up, we find a surjective group homomorphism $\R^n \to \bG$ by Pansu's differentiability theorem. Thus $\bG$, which is the graded algebra of the Malcev closure of $\pi_1(N)$, is Abelian. We then have that   $\pi_1(N)$ is virtually Abelian.

Theorem \ref{thm:LQ_main} is in connection to a question of Gromov \cite[Question 2.44]{Gromov-book} whether fundamental groups of elliptic manifolds are virtually Abelian. A Riemannian $n$-manifold $M$ is called \emph{elliptic} if there exists a Lipschitz map $\R^n \to M$ of \emph{non-zero asymptotic degree}, that is, a map $f\colon \R^n \to M$ satisfying
\[
\limsup_{r\to\infty} \frac{1}{r^n} \int_{B^n(r)}J_f > 0.
\]
By \cite[Corollary 2.43]{Gromov-book}, a closed and aspherical elliptic manifold has a virtually Abelian fundamental group. Since Lipschitz quotient maps are maps of non-zero asymptotic degree, the topological assumption on asphericality in this result can be replaced by a slightly stronger geometric assumption that the manifold admits a Lipschitz quotient map from $\R^n$.

\section{Lipschitz quotients and volume growth}

As mentioned in the introduction, Varopoulos' theorem for quasiregular maps states that the fundamental group of a closed quasiregularly elliptic manifold has growth of polynomial order at most $n$; for open quasiregularly elliptic manifolds, see \cite[Theorem 1.3]{PR}.
 
In this section, we prove an analogous result for closed, connected Riemannian manifolds admitting a Lipschitz quotient map from $\R^n$. We begin with a lifting lemma for Lipschitz quotients, which will also be useful later in the paper. 

\begin{lemma}
\label{lemma:LQ_lift}
Let $N$ be a closed, connected, and Riemannian $m$-manifold for $m\le n$, let $f\colon \R^n \to N$ be an $L$-Lipschitz quotient, and $\tilde f\colon \R^n\to N$ a lift of $f$ to the universal cover $\tilde N$ of $N$. Then $\tilde f$ is an $L$-Lipschitz quotient.
\end{lemma}
\begin{proof}
Let $\pi\colon \tilde N\to N$ be a locally isometric covering map. Since $M$ is closed, we may fix $\delta>0$ so that $\pi|B_{\tilde N}(y,\delta)\colon B_{\tilde N}(y,\delta)\to B_N(\pi(y),\delta)$ is an isometry for every $y\in \tilde N$.

Since $\tilde f$ is obviously $L$-Lipschitz, it suffices to show that $B_{\tilde N}(\tilde f(x),r/L)\subset \tilde f\left( B^n(x,r)\right)$ for each ball $B^n(x,r)$ in $\R^n$. Let $y\in B_{\tilde N}(\tilde f(x),r/L)$ and $\varepsilon\in (0,\delta)$ so that $B_{\tilde y}(y,\varepsilon)\subset B_{\tilde N}(\tilde f(x),r/L)$. 

Let $[\tilde f(x),y]$ be a geodesic in $\tilde N$ from $\tilde f(x)$ to $y$ and let $B_1,\ldots, B_k$ be a sequence of balls, where $k>1/\varepsilon$, so that $B_i=B_{\tilde N}(z_i,\varepsilon)$, where $z_i\in [\tilde f(x),y]$ and $|z_{i-1}-z_i|=|\tilde f(x)-y|/k$ for each $2\le i \le k$; we may take $z_1=\tilde f(x)$ and $z_k=y$. 

Since $f$ is a Lipschitz quotient, there exists balls $B'_i = B^n(x_i,L\varepsilon)$ in $\R^n$, for $i=1,\ldots, k$, so that $x_0=x$, $B'_{i-1}\cap B'_i\ne \emptyset$, and $fB'_i \supset \pi B_i$ for each $i$. Then $y\in \tilde fB'_k$. Since $B'_k \subset B^n(x,r+L\varepsilon)$, the claim follows.
\end{proof}

\begin{corollary}
Let $N$ be a closed, connected, and Riemannian manifold admitting a Lipschitz quotient map from $\R^n$. Then the polynomial order of growth of $\pi_1(N)$ is at most $n$. In particular, $\pi_1(N)$ is virtually nilpotent.
\end{corollary}
\begin{proof}
The claim follows directly from the volume estimate for balls in the Riemannian universal cover $\tilde N$ and Gromov's theorem. Indeed, let $f\colon \R^n \to N$ be an $L$-Lipschitz quotient.
Let $\tilde f\colon \R^n \to \tilde N$ be a lift of $f$ to the universal cover. Since the covering map $\tilde N \to N$ is a local isometry and $\tilde f$ is $L$-Lipschitz quotient,
\begin{eqnarray*}
\vol_{\tilde N} \left(B_{\tilde N}(\tilde f(x),r)\right) &\le& \vol_{\tilde N} \left(\tilde f (B^n(x,Lr)) \right) \\
&\le& \int_{B^n(x,Lr)} J_{\tilde f} \\ 
&=& \int_{B^n(x,Lr)} J_f \le L^{2n} \vol_{\R^n}(B^n(0,1)) r^n
\end{eqnarray*}
for all $x\in \R^n$ and $r>0$. Since $\pi_1(N)$ and $\tilde N$ are quasi-isometric, $\pi_1(N)$ has polynomial growth of order at most $n$. The claim follows.
\end{proof}

\section{Lipschitz quotients and ultralimits.}

We refer the reader who is not used to the following notions of nonprincipal ultrafilters and ultralimits to Chapter 9 of Kapovich's book \cite{Kapovich-book}. Roughly speaking, taking ultralimits with respect to a nonprincipal ultrafilter is a consistent way of using the axiom of choice to select an accumulation point of any bounded sequence of real numbers.
Let $\omega$  be a nonprincipal ultrafilter. Given a sequence $X_j$ of metric spaces with base points $\star_j\in X_j$, we consider the {\em based ultralimit metric space}
$$(X_\omega, \star_\omega):=(X_j, \star_j)_\omega := \lim_{j\to \omega} (X_j, \star_j).$$
We recall briefly the construction. Let 
$$X^\N_b:=\left\{ (x_j)_{j\in \N} : x_j\in  X_j , \sup\{ d(x_j, \star_j):j\in\N\}<\infty\right\}.$$
For all $ ( x_j ) _{j},  ( x'_j ) _{j}\in X^\N_b$, set
$$d_\omega ( ( x_j ) _{j},  ( x'_j ) _{j}):= \lim_{j\to \omega} d_j(x_j, x'_j),$$
where $ \lim_{j\to \omega}$ denotes  the $\omega$-limit of a sequence  indexed by $j$.
Then $X_\omega$ is the metric space obtained by taking the quotient of $(X^\N_b, d_\omega)$
by the semidistance $d_\omega$. We denote by $[x_j]$ the equivalence class of $ ( x_j ) _{j}$. The base point  $\star_\omega$ in $X_\omega$ is $[\star_j]$.

Suppose $f_j:X_j\to Y_j$ are 
maps between 
metric spaces, 
$\star_j\in X_j$
are base points,  and 
we have the property that
$  ( f_j(x_j) ) _j %{j\in \N}
 \in Y^\N_b$,
for all 
$  ( x_j ) _j %{j\in \N} 
\in X^\N_b$.
Then the ultrafilter $\omega$ assigns a limit map
$f_\omega := 
\lim_{j\to \omega} f_j :
(X_j, \star_j)_\omega 
\to
(Y_j, f_j(\star_j))_\omega $
as
$f_\omega  ([x_j]) := [f_j(x_j)]$.

In particular, if $f_j :X_j\to Y_j $ are $L$-Lipschitz maps, then 
$\lim_{j\to \omega} f_j $ is a well-defined map  
$(X_j, \star_j)_\omega 
\to
(Y_j, f_j(\star_j))_\omega$.
Moreover, 
passing to $\omega$-limits for the inequalities
$$d_j(f_j(x_j), f_j(x'_j)) \leq  L d_j(x_j, x'_j), $$
one obtains that $f_\omega$ is $L$-Lipschitz, i.e., 
$$d_\omega(f_\omega([x_j]), f_\omega([x'_j])) \leq  L d_\omega([x_j], [x'_j]). $$

\begin{lemma}\label{proof:LQ:limit}
Lipschitz quotients pass to ultralimits as Lipschitz quotients quantitatively, that is, ultralimits of $L$-Lipschitz quotients are $L$-Lipschitz quotients for all $L\ge 1$.
\end{lemma}
\begin{proof}
Suppose  that $f_j :X_j\to Y_j $ are $L$-Lipschitz quotients, i.e., in addition to be 
$L$-Lipschitz maps, we know that 
\begin{equation}
\label{LQ:extra:condition}
B_{Y_j}(f_j(x_j),r/L) \subset f_j( B_{X_j}(x_j,r)),
\end{equation}
for all $  x_j\in X_j$ and $r>0$.
Take
$  ( x_j ) _{j\in \N} \in X^\N_b$ and $  ( y_j ) _{j\in \N} \in Y^\N_b$.
Then \eqref{LQ:extra:condition} implies that,
for all 
$j\in \N$ there exists $x'_j\in X_j$ with $f_j(x'_j)=y_j$
and
$$d(x'_j, x_j) \leq L\, d(y_j, f_j(x_j)).$$
Note that
\begin{eqnarray*}
d(x'_j, \star_j)& \leq&
d(x'_j, x_j) + d(x_j, \star_j)\\
&\leq & L\, d(y_j, f_j(x_j))+ d(x_j, \star_j)\\
&\leq &  L\,d(y_j, f_j(\star_j))+  L \,d(f_j(\star_j), f_j(x_j))+ d(x_j, \star_j)\\
&\leq &  L\, d(y_j, f_j(\star_j))+ ( L^2+1)\, d(x_j, \star_j).
\end{eqnarray*}
Thus $  ( x'_j ) _{j\in \N} \in X^\N_b$.
Also $ f_\omega ([x'_j])  = [f_j(x'_j)] = [y_j]$
and
\begin{eqnarray*}
d_\omega ( [x'_j], [x_j])  &=& \lim_{j\to \omega}  d(x'_j, x_j)\\
& \leq& L\,   \lim_{j\to \omega}  d(y_j, f_j(x_j))\\
&=& L\, d_\omega( [y_j], [f_j(x_j)] )\\
&=& L\, d_\omega ( [y_j], f_\omega( [x_j] ) ).
\end{eqnarray*}
We conclude that the ultralimit map  $f_\omega$ is an $L$-Lipschitz quotient as well.
\end{proof}

Let $X$ be a metric space with distance $d_X$.
We fix a nonprincipal ultrafilter $\omega$, a base point $\star\in X$, and a sequence of positive numbers
$\lambda_j\to 0$ as $j\to \infty$.
The {\em asymptotic cone} of $X$ is defined as
$$\Cone(X):= (\lambda_j X, \star)_\omega.$$
Here 
$\lambda_j X=(X,\lambda_j d_X) $ and $\star$ is seen as the constant sequence.

From Lemma \ref{proof:LQ:limit}, we immediately have the following result.
\begin{corollary}\label{passing:cone}
Let $f: X\to Y$ be an $L$-Lipschitz quotient between metric spaces.
Then $\Cone(f):=f_\omega: \Cone(X)\to \Cone(Y)$ is 
an $L$-Lipschitz quotient.
\end{corollary}
Note that $f_\omega$ above is the limit of the sequence of maps 
$ f \colon \lambda_j X \to \lambda_j Y$, which  are  
set-wise always the same map and are 
$L$-Lipschitz quotients for all $j$.

%%  Removed in the last edit
%Similarly, quasiisometries pass to the asymptotic cones. In this case,  the asymptotic  maps  are  biLipschitz maps.

%If $X$, $\omega$, and $\star$ are as before, but  
%$\lambda_j\to \infty$, as $j\to \infty$,
%we define the {\em ultratangent at} $\star$ of $X$ as
%$${\rm Tan}(X,\star):= (\lambda_j X, \star)_\omega.$$
%Then, by Lemma \ref{proof:LQ:limit}, we also have
% the following consequence.
%\begin{corollary}\label{passing:tang}
%Let $f: X\to Y$ be an $L$-Lipschitz quotient between metric spaces.
%Then ${\rm Tan}(f, \star):=f_\omega: {\rm Tan}(X,\star)\to {\rm Tan}(Y,f(\star)%)$ is 
%an $L$-Lipschitz quotient.
%\end{corollary}

In our argument we will perform a blow down (i.e., a passing to asymptotic cones) followed by a blow up  (i.e., a passing to tangents).
In both cases we will make use of theorems by Pansu. We begin with a weaker version of Pansu's theorem on asymptotic cones sufficient for our purposes.

\begin{theorem}[Pansu, {\cite[Th\'eor\`em principal]{Pansu-croissance}}]
\label{THM:Pansu1}
Let $\Gamma$ be a nilpotent finitely generated group equipped with some word distance.
Then 
$\Cone(\Gamma)$
is a subFinsler Carnot group. In particular, $\Cone(\Gamma)$
is biLipschitz equivalent to a subRiemannian Carnot group.
\end{theorem}

\begin{remark}
\label{Abelian:IFF:Abelian}
From the section `Compl\'ement au th\'eor\`em principal' in  \cite[page 421]{Pansu-croissance},
the group structure of $\Cone(\Gamma)$ is clear:
we may assume that $\Gamma$ is a lattice in a nilpotent Lie group $G$, which is called Malcev closure of $\Gamma$.
Then $\Cone(\Gamma)$ is the graded algebra associated to $G$.
In particular, the group $\Gamma $ is virtually Abelian
if and only if
$G$ is Abelian, and 
if and only if
$\Cone(\Gamma)$ is Abelian.
\end{remark}

\begin{remark}
We refer the reader to the papers \cite{Breuillard-LeDonne1} %, Breuillard-LeDonne2} 
for another proof of Theorem \ref{THM:Pansu1} and the construction of the graded algebra and the subFinsler structure on it.
We point out that the above theorem by Pansu is stating that the asymptotic cone does not depend on the choice of ultrafilter $\omega$ nor on the scaling factors $\lambda_j$. We actually do not need such an independence. Using the theory of locally compact groups,    one can easily prove that any such an asymptotic cone is always a subFinsler group, see \cite{GromovM:Grople, Berestovskii, LeDonne_characterization}. However, for us it will be important to know that, as explained in Remark 
\ref{Abelian:IFF:Abelian}, the asymptotic cone is Abelian 
 only if
the initial group $\Gamma $ is virtually Abelian.
\end{remark}

Regarding blow-ups, we shall use the  following  differentiability theorem. In a Carnot group $\G$, we denote by $L_p\colon \G \to \G$ the left translation by $p\in \G$, and by $\delta_h\colon \G\to \G$ the dilation by $h>0$. 
\begin{theorem}[Pansu, {\cite[Th\'eor\`eme 2]{PansuP:MetCC}}]
\label{THM:Pansu2}
 Let $\G_1$ and $ \G_2$ be two subRiemannian Carnot groups.
 If $f:\G_1\to \G_2$ is Lipschitz, then  for almost every $p\in\G_1$,
the difference quotient maps $\delta_{1/h}\circ L_{f(p)}^{-1} \circ f \circ L_p \circ \delta_h$ converge uniformly on compact sets to a group homomorphism.
%   ${\rm Tan}(f,p):\G_1\to \G_2$ is a group homomorphism.
\end{theorem}
  
\begin{remark}
\label{rmk:3}
Observe that if $f_j \colon X\to Y$ are $L$-Lipschitz quotients between metric spaces and are converging uniformly on balls, then the limit is an $L$-Lipschitz quotient. Thus, in Theorem \ref{THM:Pansu2}, the group homomorphism is a Lipschitz quotient if $f$ is a Lipschitz quotient.
%Recall that, since any Carnot group $\G$ admits dilations, then ${\rm Tan}(\G)=\G$.
\end{remark}

\section{Proof of Theorem \ref{thm:LQ_main}}

By passing to the ultralimit as explained in the previous section, we obtain the following existence result. 
\begin{lemma}
\label{lemma:LQ_to_cone}
Let $f \colon \R^n \to N$ be a Lipschitz quotient  into a closed and connected Riemannian $n$-manifold. Then there exists a Lipschitz quotient  $F\colon \R^n \to \Cone(\pi_1(N))$ to the asymptotic cone of $\pi_1(N)$.
\end{lemma}
\begin{proof}
Let $\tilde N $ be the universal cover of $N$ and $\tilde f\colon \R^n \to \tilde N$ a lift of $f$. By Lemma \ref{lemma:LQ_lift}, $\tilde f$ is a Lipschitz quotient. 

Then, by Corollary  \ref{passing:cone}, there exists a Lipschitz quotient $\tilde f_\omega \colon \R^n \to \Cone(\tilde N)$; here we use the fact that the asymptotic cone of the Euclidean space $\R^n$ is just $\R^n$ itself.

Let $S$ be a finite symmetric generating set of $\pi_1(N)$. 
Equip $\pi_1(N)$ with the word distance associated to $S$.
Since $N$ is closed, the metric spaces  $\tilde N$ and  $\pi_1(N)$ are quasi-isometric. Thus, by passing to asymptotic cones, we get      a biLipschitz homeomorphism $\phi \colon \Cone(\tilde N) \to \Cone(\pi_1(N))$. Thus  $F=\phi \circ f_\omega\colon \R^n \to \Cone(\pi_1(N))$ is the desired map.
\end{proof}

\begin{proof}[Proof of Theorem \ref{thm:LQ_main}]
Let $N$ be a closed, connected Riemannian $n$-manifold admitting a Lipschitz quotient  $f\colon \R^n \to N$. Then, 
by Lemma \ref{lemma:LQ_to_cone} there
 exists a Lipschitz quotient  $F\colon \R^n \to \Cone(\pi_1(N))$ to the asymptotic cone of $\pi_1(N)$.
  By Theorem \ref{THM:Pansu1}, $\Cone(\pi_1(N))$ is bilipschitz equivalent to a Carnot group $\bG$.

From Theorem \ref{THM:Pansu2}, since $F$ is a Lipschitz map between Carnot groups, we have that,  for almost every $p\in\R^n$, the tangent map $\R^n\to \G$  is a group homomorphism. Let $\phi$ be any of such maps. Since $\phi$ is an ultratangent of a  Lipschitz quotient, it is a Lipschitz quotient by Remark \ref{rmk:3}. In particular, $\phi$ is surjective. Thus $\bG$ is the image under a homomorphism of the Abelian group $\R^n$. Hence, the group $\bG$ is Abelian.
 Thus $\pi_1(N)$ is virtually Abelian by Remark \ref{Abelian:IFF:Abelian}.
\end{proof}

\bibliographystyle{abbrv}
\bibliography{Rigidity_paper}

\end{document}